\documentclass[10pt]{article}
\usepackage[english]{babel}
\usepackage{amsmath,amsthm}
\usepackage{amsfonts}
\usepackage[none]{hyphenat}
\usepackage{xcolor}
\usepackage[font={small}]{caption}
\usepackage[explicit]{titlesec}
\usepackage{anyfontsize}
\titleformat{\section}{\Large}{\textbf{\thesection .}}{1em}{\textbf{#1}}

\addto{\captionsenglish}{}
\newtheorem{thm}{Theorem}[section]
\newtheorem{cor}[thm]{Corollary}

\newtheorem{prop}[thm]{Proposition}
\theoremstyle{definition}
\newtheorem{defn}[thm]{Definition}
\theoremstyle{remark}

\numberwithin{equation}{section}

\renewenvironment{proof}{{\vspace{-1em}\quad \textit{Proof.}}}{\hfill$\square$}
\begin{document}
\renewcommand{\thefootnote}{\fnsymbol{footnote}}
\title{Forbidden Substrings In Circular $K$-Successions}%
\author{\vspace{-3,5\baselineskip}Enrique Navarrete$^{\ast}$}\footnotetext[1]{Grupo ANFI, Universidad de Antioquia.}%
%\address{}%
%\thanks{}%
\date{}%
\renewcommand*{\thefootnote}{\arabic{footnote}}
% ----------------------------------------------------------------
\makeatletter
\def\maketitle{%
\bgroup
\par\centering{\LARGE\@title}\\[3em]%
\ \par
{\@author}\\[4em]%
\egroup
}

\maketitle
\begin{abstract}
    In this note we define a circular $k$-succession in a permutation $p$ on $[n]$ as either a pair $p(i)$, $p(i + 1)$ if $p(i + 1) \equiv p(i) + k\ (\text{mod }n)$, or as the pair $p(n)$, $p(1)$ if $p(1) \equiv p(n) + k\ (\text{mod }n)$.  We count the number of permutations that for fixed $k$, $k < n$, avoid substrings $j(j+k)$,\linebreak \mbox{$1 \leq j \leq n-k$}, as well as permutations that avoid substrings\linebreak $j(j+k)\ (\text{mod }n)$ for all $j$, $1 \leq j \leq n$.  We also count circular permutations that avoid such substrings, and show that for substrings $j(j+k)\ (\text{mod }n)$, the number of permutations depends on whether $n$ is prime, and more generally, on whether $n$ and $k$ are relatively prime.\\[1em]
\end{abstract}
% -----------------------------------------------section 1---------------------------------------------------------------------------------
\section{Introduction and Previous Results}
In [3] we counted the number of permutations according to $k$-shifts, where we defined $\{d^k_n\}$ as the set of permutations on $[n]$ that for fixed $k$, $k < n$, avoid substrings $j(j+k)$, $1 \leq j \leq n-k$.  The number of such permutations, $d^k_n$, turned out to be
\begin{equation}\label{equ11}
  d^k_n=\sum^{n-k}_{j=0}(-1)^j\binom{n-k}{j}(n-j!).
\end{equation}

The forbidden substrings in these permutations can be pictured as a diagonal running $k$ places to the right of the main diagonal of an $n\times n$ chessboard (hence the term \textquotedblleft $k$-shifts\textquotedblright). Note that we will also refer to these forbidden substrings that are $k$ spacings apart as \textquotedblleft $k$-successions\textquotedblright; therefore in [3] we counted permutations with no $k$-successions for $k > 1$ (the case $k = 1$ was discussed in [2]).  Note that there are some references that not only count permutations with no successions but also count permutations with\linebreak $i = 1, 2, \ldots, n-1$ successions (the so-called \textquotedblleft succession numbers\textquotedblright) but they do it only for the case $k = 1$, \textit{ie}. for forbidden substrings that are only one spacing apart, $j(j+1)$, so some care should be taken (see [1], for example).

Table \ref{tab1} in the Appendix provides some $d^k_n$ values.  For example, for $n = 4$, $k = 2$, the forbidden substrings that are 2 spacings apart are $\{13, 24\}$, and $d^2_4= 14$ since there are 14 permutations in $S_4$ that avoid such substrings or 2-successions.

The sequence $\langle d^k_n\rangle$ is available in OEIS. For example, for $k = 4$, the sequence is now A277563 [4], and for $k = 1$ it is A000255 [5]. Note that for $k = 0$, $d^k_n= Der(n)$, the $n$th derangement number, which is A000166 in OEIS [9].
	
In [3] we also counted the number of permutations according to $k$-shifts $(\text{mod }n)$, where we defined $\{D^k_n\}$ as the set of permutations on $[n]$ that for fixed $k$, $k < n$, avoid substrings $j(j+k)$ for $1 \leq j\leq n-k$, and avoid substrings $j(j+k)\ (\text{mod }n)$ for $n-k < j \leq n$. Note that we can summarize in the single definition \textquotedblleft avoid substrings $j(j+k)\ (\text{mod }n)$ for all $j$, $1 \leq j \leq n$\textquotedblright\  if we agree to write $n$ instead of 0 when doing addition $(\text{mod }n)$.

The forbidden substrings in this case are easily seen along an $n\times n$ chessboard, where for $j > n - k$, the forbidden positions start again from the first column along a diagonal $n-k$ places below the main diagonal, as in Figure 1 below.

\begin{figure}[h!]
  \centering
  \begin{tabular}{c|c|c|c|c|}
    \multicolumn{1}{c}{} & \multicolumn{1}{c}{1} & \multicolumn{1}{c}{2} & \multicolumn{1}{c}{3} & \multicolumn{1}{c}{4} \\
    \cline{2-5}
    1 &  &  &  & $\times$ \\
    \cline{2-5}
    2 & $\times$ &  &  &  \\
    \cline{2-5}
    3 &  & $\times$ &  &  \\
    \cline{2-5}
    4 &  &  & $\times$ &  \\
    \cline{2-5}
  \end{tabular}
  \\
  \caption{Forbidden positions in $\{D^3_4\}$.}\label{fig1}
\end{figure}

Figure \ref{fig1} shows forbidden positions on a $4\times 4$ chessboard that correspond to forbidden substrings of permutations in $\{D^3_4\}$. These forbidden substrings are $\{14; 21, 32, 43\}$. The forbidden substrings below the diagonal are separated by a semicolon; these are the forbidden substrings $j(j+k)\ (\text{mod }n)$ for $n-k < j \leq n$. Note that while there are only $n - k$ forbidden substrings in $\{d^k_n\}$, there are $n$ forbidden substrings in $\{D^k_n\}$.

The number of permutations in $\{D^k_n\}$ is given by
\begin{equation}\label{equ22}
  D^k_n=\sum^{n-1}_{j=0}(-1)^j\binom{n}{j}(n-j)!.
\end{equation}
Furthermore, there are the same number of permutations in $\{D^k_n\}$ as in $\{D_n\}$ whenever $n$ and $k$ are relatively prime ($D_n =D^1_n$; see [3]).  Table \ref{tab4} in the Appendix provides a table for some $D^k_n$ values. For example, for $n = 4$, $k = 3$, we have seen that forbidden substrings in $\{D^3_4\}$ are $\{14; 21, 32, 43\}$ and $D^3_4= 4$ since there are 4 permutations in $S_4$ that avoid such substrings or 3-successions. Furthermore, since $(4, 3) = 1$, there is the same number of permutations in $\{D^3_4\}$ as in $\{D_4\}$, which are permutations in $S_4$ that avoid successions $\{12, 23, 34; 41\}$.  The sequence $\langle D_n\rangle$ is A000240 in OEIS [6].
\newpage
In this note we will discuss permutations in $\{d^k_n\}$ and $\{D^k_n\}$ with no \textit{circular }$k$-successions, as described in the following definition.
\begin{defn}\label{def1}
We define a $circular\ k\text{-}succession$ in a permutation $p$ on $[n]$ as either a pair $p(i)$, $p(i + 1)$ if $p(i + 1) \equiv p(i) + k\ (\text{mod }n)$, or as the pair $p(n)$, $p(1)$ if $p(1) \equiv p(n) + k\ (\text{mod }n)$. We will denote as $\{{d^\ast}^k_n\}$ the set of permutations in $\{d^k_n\}$ that have no circular $k$-successions, and similarly, as $\{{D^\ast}^k_n\}$ the set of permutations in $\{D^k_n\}$ that have no circular $k$-successions. These permutations are in one-line notation.
\end{defn}
Hence permutations in $\{{d^\ast}^k_n\}$ avoid substrings $j(j+k)$, $1 \leq j \leq n-k$, and permutations in $\{{D^\ast}^k_n\}$ avoid substrings $j(j+k)\ (\text{mod }n)$ for all $j$, $1 \leq j \leq n$ (as in $\{d^k_n\}$ and $\{D^k_n\}$, respectively), but in $\{{d^\ast}^k_n\}$ and $\{{D^\ast}^k_n\}$ there is the additional restriction of the circular $k$-successions.

For example, in $\{D^3_4\}$, we saw above that the forbidden substrings are\linebreak $\{14; 21, 32, 43\}$, hence the permutation 2134 has a circular 3-succession 21 and the permutation 2413 has a circular 3-succession 32.  Hence such permutations won't be allowed in $\{{D^\ast}^3_4\}$ (but 2413 is allowed in $\{D^3_4\}$). Note as mentioned above that the forbidden substrings in
$\{{D^\ast}^k_n\}$ are the same as those in $\{D^k_n\}$ and the same is true for $\{{d^\ast}^k_n\}$ and $\{d^k_n\}$, with differences in cardinalities being accounted by circular $k$-successions, as will be shown below.

Note that in Definition \ref{def1} we don't really need the definition $(\text{mod }n)$ in $\{{d^\ast}^k_n\}$ since forbidden positions in those permutations stay above the main diagonal of an $n\times n$ chessboard.
% -----------------------------------------------section 2---------------------------------------------------------------------------------
\section{Main Lemmas and Propositions}
In [2] we divided permutations according to their starting digit and called these groups \textit{classes}. We showed that permutations in $\{D_n\}$ are equidistributed, meaning that all classes are equinumerous. This is also true for deranged permutations (see [2]).

Now we note that in general $\{{d^\ast}^k_n\}$, which forbids permutations with circular $k$-successions, will have less permutations than $\{d^k_n\}$.  However, if we look at individual classes starting with digits $1, 2, \ldots, n$, we see that there will be classes in $\{d^k_n\}$ and $\{{d^\ast}^k_n\}$ that have the same number of elements, hence they are the same.

Indeed, since forbidden substrings in $\{d^k_n\}$ run strictly over the main diagonal, there will be substrings $j(j+k)$ where $(j+k) \neq i$, $i = 1, 2, \ldots, k$, hence $\{{d^\ast}^k_n\}$ will not have circular $k$-successions starting with $i = 1, 2, \ldots, k$. Hence classes in $\{d^k_n\}$ and $\{{d^\ast}^k_n\}$ starting with these digits will be the same, and classes starting with $i = k + 1, \ldots, n$ will be smaller in $\{{d^\ast}^k_n\}$ than in $\{d^k_n\}$. Hence a total of $k$ classes will be the same in $\{{d^\ast}^k_n\}$ and in $\{d^k_n\}$, and $n-k$ classes will be smaller in $\{{d^\ast}^k_n\}$ than in $\{d^k_n\}$, where $n-k$ is the number of forbidden substrings in $\{d^k_n\}$ and $\{{d^\ast}^k_n\}$.

For example, for $n = 4$, $k = 3$, the forbidden substring in $\{d^3_4\}$ is $\{14\}$, so there won't be permutations in $\{{d^\ast}^3_4\}$ ending in 1 and starting with 4, since they would represent circular 3-successions, which are forbidden. Hence the class of permutations starting with 4 is smaller in $\{{d^\ast}^3_4\}$ than in $\{d^3_4\}$ and the other classes starting with $i = 1, 2, 3$ are equal in both cases. In fact, the cardinalities of the four classes in $\{{d^\ast}^3_4\}$ are $\{4, 4, 4, 4\}$, while the cardinalities of the four classes in $\{d^3_4\}$ are $\{4, 4, 4, 6\}$, the difference being accounted by the permutations 4231 and 4321 which are allowed in $\{d^3_4\}$ but not in $\{{d^\ast}^3_4\}$.	

Similarly, for $n = 6$, $k = 2$, forbidden substrings in $\{d^2_6\}$ are $\{13, 24, 35, 46\}$, so there will be forbidden circular 2-successions in $\{{d^\ast}^2_6\}$ starting with 3, 4, 5, 6 but not with 1 and 2. Therefore, the classes in $\{d^2_6\}$ and in $\{{d^\ast}^2_6\}$ starting with 1 and 2 will be the same, and the other classes will be smaller in $\{{d^\ast}^2_6\}$ than in $\{d^2_6\}$. In fact, the cardinalities of the six classes in $\{d^2_6\}$ are $\{53, 53, 64, 64, 64, 64\}$, for a total of 362 permutations, while the cardinality of the classes in $\{{d^\ast}^2_6\}$ is 53 in all six cases, for a total of 318.  Hence there are 318 permutations in $S_6$ that have no circular 2-successions, \textit{ie}. permutations that avoid substrings such as 13 in 213456 or in 345621.

From these examples one might suspect that permutations in $\{{d^\ast}^k_n\}$ are\linebreak equidistributed (equinumerous), a conjecture which is confirmed in the next section.
\subsection{Results for permutations in $\boldsymbol{\{{d^\ast}^k_n\}}$}
\begin{prop}\label{propo21}
 For fixed $k$, $1 \leq k \leq n-1$, if ${d^\ast}^k_n$ denotes the number of permutations on $[n]$ that avoid circular $k$-successions, then ${d^\ast}^k_n= nd^{k-1}_{n-1}$.
\end{prop}

\begin{proof}
Consider a class in $\{{d^\ast}^k_n\}$ where both $\{{d^\ast}^k_n\}$ and $\{d^k_n\}$ have the same number of elements (there is at least one such class for $1 \leq k \leq n-1$), and denote the cardinality of such a class by $c$. Since a circular permutation is an $n$-to-1 mapping with respect to a linear one, we have that ${d^\ast}^k_n= n\cdot c$. By inclusion-exclusion, it is straightforward to show that
\begin{equation}\label{equ23}
  c=\sum^{n-k}_{j=0}(-1)^j\binom{n-k}{j}(n-j-1)!
\end{equation}
since the combinatorial term counts the number of ways to get substrings of length $j$ while the term $(n - j - 1)!$ counts the $(n - j)!/(n - j)$ circular permutations of forbidden substrings of length $j$ and the remaining elements. Hence, by Equation (\ref{equ11}), we see that the right-hand side of Equation (\ref{equ23}) is just $d^{k-1}_{n-1}$ and the proof follows.
\end{proof}

For example, in $\{d^2_4\}$ we see that the forbidden substrings are $\{13, 24\}$.  Recall from [3] that we define a minimal forbidden substring as two consecutive elements $ik$. We assign to this minimal substring a length equal to one. Recall also that a forbidden substring of length $j$ can be considered as either a single substring of length $j$ or as $j$ overlapping substrings of length 1.

Hence for the term $j = 0$ in Equation \ref{equ23} we see that we can choose no forbidden substrings in $\binom{2}{0}$ ways and we can permute them circularly with the remaining elements 1,2,3,4 in $4!/4$ ways. For the term $j = 1$ we can choose 1 forbidden substring in $\binom{2}{1}$ ways and we can permute it circularly with the remaining elements in $3!/3$ ways (for example, choose the substring 13 and permute the blocks 13, 2, 4 circularly in $3!/3$ ways). For the term $j = 2$ we can choose 2 forbidden substrings in $\binom{2}{2}$ ways and we can permute them circularly with the remaining elements in $2!/2$ ways (that is, permute the blocks 13, 24 circularly in $1!$ way). Then we have that ${d^\ast}^k_n= n\cdot c = 4\cdot 3$, so there are 12 permutations in $S_4$ that avoid substrings $\{13, 24\}$ in circular 2-successions (for example, permutations such as 1342 and 3421 are not allowed).

Note that even though permutations in $\{d^k_n\}$ are not equidistributed, we gain this property when we forbid circular $k$-successions in $\{d^k_n\}$, as recorded in the following corollary, which follows from Proposition \ref{propo21}, ${d^\ast}^k_n=nd^{k-1}_{n-1}$.
\begin{cor}\label{coro22}
The classes in $\{{d^\ast}^k_n\}$ are equidistributed \textup{(}equinumerous\textup{)} and the cardinality of each class is given by $d^{k-1}_{n-1}$.
\end{cor}
Note that the property of classes being equidistributed is not shared by all types of permutations with forbidden substrings. For example, $\{d_n\}$ and $\{d^k_n\}$ don't have this property, while $\{D_n\}$ and $\{D^k_n\}$ do.

Now we define ${c^\ast}^k_n$ as the number of circular permutations on $[n]$ (in cycle notation) that avoid substrings $j(j+k)$, $1 \leq j \leq n-k$, \textit{ie}. the same forbidden substrings as in $\{d^k_n\}$ and $\{{d^\ast}^k_n\}$.  We have the following corollary that counts such permutations.
\begin{cor}\label{coro23}
For fixed $k$, $1 \leq k \leq n-1$, the number of circular permutations ${c^\ast}^k_n$ on $[n]$ that avoid substrings $j(j+k)$, $1 \leq j \leq n-k$, is given by
\begin{equation}\label{equ24}
  {c^\ast}^k_n=\sum^{n-k}_{j=0}(-1)^j\binom{n-k}{j}(n-j-1)!.
\end{equation}
\end{cor}
\begin{proof}
This is the number $c$ from the proof of Proposition \ref{propo21}. Hence we may write ${d^\ast}^k_n=n{c^\ast}^k_n$ , where ${c^\ast}^k_n=d^{k-1}_{n-1}$.
\end{proof}
\newpage
For example, ${c^\ast}^2_5=d^1_4= 11$, so there are 11 circular permutations in $\{{c^\ast}^2_5\}$ that avoid substrings $\{13, 24, 35\}$. Furthermore, there are $5\cdot 11 = 55$ permutations in $\{{d^\ast}^2_5\}$ that avoid such substrings in circular 2-successions.

Similarly, ${c^\ast}^2_4=d^1_3= 3$, so there are 3 circular permutations that avoid\linebreak substrings $\{13, 24\}$. These are the permutations $\{(1234), (1423), (1432)\}$, and since $n = 4$, to each of these circular permutations correspond 4 permutations in $\{{d^\ast}^2_4\}$ with no circular 2-succession (for example the ones for (1234) are\linebreak $\{1234, 2341, 3412, 4123\}$).
\begin{cor}\label{coro24}
The number of circular permutations on $[n]$ that avoid substrings $j(j+1)$, $1 \leq j \leq n-1$, \textup{(}\textit{ie}. $k = 1$\textup{)} is given by
\begin{equation}\label{equ25}
  {c_n}^\ast=Der(n-1),
\end{equation}
where $Der(n)$ is the $n$th derangement number.
\end{cor}
\begin{proof}
Since we have that
\begin{equation}\label{equ26}
  Der(n)=\sum^{n}_{j=0}(-1)^j\binom{n}{j}(n-j)!
\end{equation}
and since Equation \ref{equ24} holds for $k = 1$, the result is immediate $({c^\ast}^1_n={c^\ast}_n)$.
\end{proof}

For example, for $n = 3$, $k = 1$, the number of circular permutations that avoid substrings $\{12, 23\}$ is $Der(2) = 1$ and this circular permutation is $(132)$, which corresponds to the 3 permutations $\{132, 213, 321\}$ in $\{{d^\ast}_3\}$.

Tables \ref{tab1} - \ref{tab3} in the Appendix show some values for $d^k_n$, ${d^\ast}^k_n$, and ${c^\ast}^k_n$.

\subsection{Results for permutations in $\boldsymbol{\{{D^\ast}^k_n\}}$}

As in the case of $D^k_n$ , the numbers ${D^\ast}^k_n$ are more difficult to get. As opposed to permutations in $\{{d^\ast}^k_n\}$, permutations in $\{{D^\ast}^k_n\}$ avoid circular $k$-successions starting with all digits since a total of $n$ substrings are forbidden (as opposed to $n-k$ forbidden substrings in $\{{d^\ast}^k_n\}$).

We first count the number of permutations in $\{D^\ast_n\}$, (\textit{ie}. $k =1$). We have the following proposition.

\begin{prop}\label{propo25}
The number of permutations in $\{{D^\ast}_n\}$ is given by
\begin{equation}\label{equ27}
  {D^\ast}_n=n\left[\sum^{n-1}_{j=0}(-1)^j\binom{n}{j}(n-j-1)!+(-1)^n\right].
\end{equation}
\end{prop}

\begin{proof}
As in the proof of Proposition \ref{propo21}, we first obtain the number of circular permutations without forbidden substrings. It is easy to see that the number of such permutations is given by
\begin{equation}\label{equ28}
  C=\sum^{n-1}_{j=0}(-1)^j\binom{n}{j}(n-j-1)!+(-1)^n
\end{equation}
since the combinatorial term counts the number of ways to get forbidden\linebreak substrings while the term $(n - j-1)!$ counts the $(n-j)!/(n-j)$ circular permutations of forbidden substrings and the remaining elements. The result is then obtained by inclusion-exclusion, where for the term $j = n$, we note that the only way to get an $n$-element substring is by the circular permutation $(12\ldots n)$.
\end{proof}

For example, the forbidden substrings in $\{{D^\ast}_4\}$ are $\{12, 23, 34; 41\}$. It is easy to count, for instance, that there are $\binom{4}{2}$ forbidden substrings of length $j = 2$ and $2!/2$ circular permutations of these substrings and the remaining elements. For example, a substring of length 2 is given by 123 and we count\linebreak $2!/2 = 1!$ circular permutation of the blocks 123, 4. Another substring of length 2 (alternatively, two substrings of length 1) is given by 12 34 and we also count $2!/2 = 1!$ circular permutation of these two blocks.
Furthermore, the only way to get $j = 4$ forbidden substrings is by the circular permutation $(1234)$, which produces the forbidden substrings $\{12, 23, 34; 41\}$. Hence $C = 1$ by Equation \ref{equ28} and by Equation \ref{equ27} we have that ${D^\ast}_4= 4$, which counts the permutations $\{1432, 2143, 3214, 4321\}$ in $S_4$ that have no circular succession, (\textit{ie}. no circular 1-succession) and avoid substrings $\{12, 23, 34; 41\}$.

If we now move on to permutations without $k$-successions that avoid substrings $j(j+k)\ (\text{mod }n)$ for $1 \leq j \leq n$, we see that as in the case of $\{D^k_n\}$, the number of permutations in $\{{D^\ast}^k_n\}$ depends on whether $n$ is prime, and more generally, on whether $n$ and $k$ are relatively prime. Since the proof follows closely the one for $\{D^k_n\}$, we refer the reader to [3] for the details.

We will only record the main counting result in the proposition below.

\begin{prop}\label{propo26}
The number of   permutations in $\{{D^\ast}^k_n\}$ with $n$ relative prime to $k$, $n \geq 3$, $k < n$, is the same as the number of permutations in $\{{D^\ast}_n\}$.	
\end{prop}

The key result in the proof of the proposition is that, if $(n,k) = 1$, there are forbidden substrings of all lengths in $\{D^k_n\}$ (see [3]).  But forbidden substrings are the same in $\{D^k_n\}$ and in $\{{D^\ast}^k_n\}$, so in this case there will also exist forbidden substrings of all lengths in $\{{D^\ast}^k_n\}$.

The proposition implies the following corollary:

\begin{cor}\label{coro27}
The number of permutations in $\{{D^\ast}^k_p\}$, $k = 1, 2, \ldots,  p-1$, is the same as the number of permutations in $\{{D^\ast}_p\}$, $p$ prime, $p \geq 3$.
\end{cor}

\begin{proof}
$(p,k) = 1$, $k = 1, 2, \ldots, p-1$.
\end{proof}

Note that Proposition \ref{propo26} is not true if $n$ is not relative prime to $k$, for example in $\{{D^\ast}^2_4\}$. In this case the forbidden substrings are $\{13, 24; 31, 42\}$, and we can't get substrings of lengths 3 or 4. On the other hand, since $(4,3) = 1$, we see that forbidden substrings in $\{{D^\ast}^3_4\}$ are $\{14; 21, 32, 43\}$, and we can get substrings of all lengths. For example, 2143 is a substring of length 3 (recall that the unit forbidden substring $ik$ has length one, hence the length of forbidden substrings will be one less than the number of elements), and a substring of length 4 (alternatively, 4 substrings of length 1) is given by the circular permutation (1432). We then have that \mbox{${D^\ast}_4=4$} by Equation \ref{equ27} and hence ${D^\ast}^3_4=4$ by Proposition \ref{propo26}.

As in the case of $\{{d^\ast}^k_n\}$, permutations in $\{{D^\ast}^k_n\}$ are equidistributed, as established by the following corollary:\\[0.1em]
\begin{cor}\label{coro28}
The classes in $\{{D^\ast}^k_n\}$ are equidistributed \textup{(}equinumerous\textup{)} and the cardinality of each class is given by the term $C$ in Equation \ref{equ28}.
\end{cor}

\begin{proof}
Equation \ref{equ27} in Proposition \ref{propo25}.
\end{proof}

We now define ${C^\ast}^k_n$ as the number of circular permutations on $[n]$ (in cycle notation)  that avoid substrings $j(j+k)\ (\text{mod }n)$ for $1 \leq j \leq n$, \textit{ie}. the same substrings as in $\{D^k_n\}$ and $\{{D^\ast}^k_n\}$. We have the following corollary that counts such permutations for $n$ and $k$ relatively prime.\\[0.1em]
\begin{cor}\label{coro29}
For $(n,k) = 1$, the number of circular permutations ${C^\ast}^k_n$ on $[n]$ that avoid substrings $j(j+k)\ (\text{mod }n)$ for $1 \leq j \leq n$, is given by
\begin{equation}\label{equ29}
  {C^\ast}^k_n=\sum^{n-1}_{j=0}(-1)^j\binom{n}{j}(n-j-1)!+(-1)^n.
\end{equation}
\end{cor}

\begin{proof}
This is the number $C$ from the proof of Proposition \ref{propo25}. Since forbidden substrings in $\{{C^\ast}^k_n\}$ are the same as those in $\{{D^\ast}^k_n\}$ and since\linebreak $(n,k) = 1$, there are forbidden substrings of all lengths $j =0, 1, 2, \ldots, n-1,n$,  so the sum in Equation \ref{equ29} applies.
\end{proof}

For example, since $(4,3) = 1$, we use Equation \ref{equ29} to compute ${C^\ast}^3_4 = 1$, so there is only one circular permutation that avoids substrings $\{14; 21, 32, 43\}$. This circular permutation is $(1234)$, which corresponds to the four permutations $\{1234, 2341, 3412, 4123\}$ in $\{{D^\ast}^3_4\}$ with no circular 3-succession. Since $(4,1) = 1$, we also have ${C^\ast}^1_4 ={C^\ast}_4 = 1$ by Equation \ref{equ29}, and the circular permutation that avoids substrings $\{12, 23, 34; 41\}$ is $(1432)$, which corresponds to the\linebreak permutations $\{1432, 2143, 3214, 4321\}$ in $\{{D^\ast}_4\}$, as seen above.

We have that the \textit{gcd} of $n$ and $k$ is very important to determine the number of permutations in $\{{C^\ast}^k_n\}$.  In fact, if for the same $n$ we have that $(n, k_1) = (n, k_2)$, then
$\{{C^\ast}^{k_1}_n\}$ and $\{{C^\ast}^{k_2}_n\}$ will have the same number of permutations. The same is true for $\{{D^\ast}^k_n\}$.

Note that usually ${C^\ast}^k_n\neq {c^\ast}^k_n$ except for the case $n = 3$ (and trivially for $n = 2$). Indeed, we have that ${C^\ast}_3={c^\ast}_3= 1$ since both numbers count the permutation (132), and ${C^\ast}^2_3={c^\ast}^2_3= 1$ since both numbers count the permutation $(123)$. This in turn implies that ${D^\ast}_3={d^\ast}_3= 3$ and ${D^\ast}^2_3={d^\ast}^2_3= 3$.

As a final remark for the case $k = 1$ in circular permutations, recall that\linebreak ${c^\ast}^1_n={c^\ast}_n$, and this counts the number of circular permutations that avoid substrings  $j(j+1)$, $1 \leq j \leq n-1$; \textit{ie}. substrings ${12, 23, \ldots, (n-1)n}$. By Corollary \ref{coro24}, ${c_n}^\ast= Der(n-1)$, so for example, for $n = 4$, we have that ${c^\ast}_4= Der(3) = 2$, and the 2 permutations that avoid such substrings are $\{(1324), (1432)\}$.

On the other hand, ${C^\ast}^1_n={C^\ast}_n$, and this counts the number of permutations that avoid substrings $j(j+1)\ (\text{mod }n)$ for $1 \leq j \leq n$; \textit{ie}. substrings\linebreak $\{12, 23, \ldots, (n-1)n, n1\}$. By Corollary \ref{coro29}, since $(4,1) = 1$, we have seen that ${C^\ast}^1_4 = 1$ and the single permutation that avoids such substrings is $(1432)$ (the permutation $(1324)$ from $\{{c^\ast}_4\}$ is excluded since it has the succession 41, which is forbidden in $\{{C^\ast}_4\}$).

Note that while in some references such as [1] ${C^\ast}_n$ is referred to as \textquotedblleft the number of circular permutations without a succession\textquotedblright, we also consider circular\linebreak permutations without a succession in $\{{c^\ast}_n\}$ on the smaller subset of forbidden substrings $\{12, 23, \ldots, (n-1)n\}$. Furthermore, the same reference counts circular permutations with $i = 1, 2, \ldots, n-1$ successions (the so-called \textquotedblleft circular\linebreak succession numbers\textquotedblright) but it does it only for forbidden substrings that are one spacing apart, $j(j+k)$, $k=1$.

For $k > 1$, this note generalizes the enumeration of circular permutations without $k$-successions (or $k$-shifts) for both kinds of forbidden substrings (above the diagonal of an $n\times n$ chessboard in $\{{c^\ast}^k_n\}$ and both above and below the diagonal, \textit{ie}. $(\text{mod }n)$ in $\{{C^\ast}^k_n\}$). It also enumerates the corresponding\linebreak permutations in one-line notation in $\{{d^\ast}^k_n\}$ and $\{{D^\ast}^k_n\}$.

For further references, the sequence $\langle {C^\ast}_n\rangle$ is A000757 in OEIS [7], $\langle {D^\ast}_n\rangle$ is A167760 [8], $\langle {d^\ast}_n\rangle$ is A000240 [6], and $\langle {c^\ast}_n\rangle$ can be looked up in the derangement numbers A000166 [9] due to Corollary \ref{coro24}. Note that A000240 not only counts the number of permutations of $[n]$ having no circular succession, ${d^\ast}_n$, but \linebreak also the number of permutations $D_n$ on $[n]$ having no substring in\linebreak $\{12, 23,  \ldots, (n-1)n, n1\}$, as well as the number of permutations of $[n]$ having exactly one fixed point (see [2]).

Tables \ref{tab4} - \ref{tab6} in the Appendix show some values for $D^k_n$, ${D^\ast}^k_n$, and ${C^\ast}^k_n$.

\newpage
\section*{References}

{\small
[1] C. Charalambides, Enumerative Combinatorics, CRC Press, 2002.

[2] E. Navarrete, Forbidden Patterns and the Alternating Derangement Sequence,
\hspace*{0.5cm}arXiv:1610.01987 [math.CO], 2016.

[3] E. Navarrete, Generalized $K$-Shift Forbidden Substrings in Permutations, \\
\hspace*{0.5cm}arXiv:1610.06217 [math.CO], 2016.

[4] N. J. A. Sloane, The On-Line Encyclopedia of Integer Sequences, published electronically
\hspace*{0.5cm}at http://oeis.org, sequence A277563.

[5] N. J. A. Sloane, The On-Line Encyclopedia of Integer Sequences, published electronically
\hspace*{0.5cm}at http://oeis.org, sequence A000255.

[6] N. J. A. Sloane, The On-Line Encyclopedia of Integer Sequences, published electronically
\hspace*{0.5cm}at http://oeis.org, sequence A000240.

[7] N. J. A. Sloane, The On-Line Encyclopedia of Integer Sequences, published electronically
\hspace*{0.5cm}at http://oeis.org, sequence A000757.

[8] N. J. A. Sloane, The On-Line Encyclopedia of Integer Sequences, published electronically
\hspace*{0.5cm}at http://oeis.org, sequence A167760.

[9] N. J. A. Sloane, The On-Line Encyclopedia of Integer Sequences, published electronically
\hspace*{0.5cm}at http://oeis.org, sequence A000166.

}
%-----------------------------------------------------------------------------------------------------
%-----------------------------------------------------------------------------------------------------
%-----------------------------------------------------------------------------------------------------
%\newpage
\begin{center}
{\Large \textbf{APPENDIX}}
\end{center}
\appendix
\begin{table}[h!]
  \centering
  \begin{tabular}{|c|r|r|r|r|r|r|}
     \hline
      \hspace{0.37cm}$n$\hspace{0.37cm} & \hspace{0.18cm}$Der_n$\hspace{0.18cm} & \hspace{0.38cm}$d_n$\hspace{0.38cm} & \hspace{0.38cm}$d^2_n$\hspace{0.38cm} & \hspace{0.38cm}$d^3_n$\hspace{0.38cm} & \hspace{0.38cm}$d^4_n$\hspace{0.38cm} & \hspace{0.38cm}$d^5_n$\hspace{0.38cm} \\
     \hline
     1& 0& & & & &\\
     2& 1& 1& & & &\\
     3& 2& 3& 4& & &\\
     4& 9& 11& 14& 18& &\\
     5& 44& 53& 64& 78& 96&\\
     6& 265& 309& 362& 426& 504& 600\\
     7& 1,854& 2,119& 2,428& 2,790& 3,216& 3,720\\
     8& 14,833& 16,687& 18,806& 21,234& 24,024& 27,240\\
     \hline
   \end{tabular}
  \caption{Some values for $d_n^k$.}\label{tab1}
\end{table}

\begin{table}[h!]
  \centering
  \begin{tabular}{|c|r|r|r|r|r|r|}
     \hline
      \hspace{0.37cm}$n$\hspace{0.37cm} & \hspace{0.18cm}${d_n}^\ast$\hspace{0.18cm} & \hspace{0.38cm}${d^\ast}^2_n$\hspace{0.38cm} & \hspace{0.38cm}${d^\ast}^3_n$\hspace{0.38cm} & \hspace{0.38cm}${d^\ast}^4_n$\hspace{0.38cm} & \hspace{0.38cm}${d^\ast}^5_n$\hspace{0.38cm} & \hspace{0.38cm}${d^\ast}^6_n$\hspace{0.38cm} \\
     \hline
     1& & & & & &\\
     2& 0& & & & &\\
     3& 3& 3& & & &\\
     4& 8& 12& 16& & &\\
     5& 45& 55& 70& 90& &\\
     6& 264& 318& 384& 468& 576&\\
     7& 1,855& 2,163& 2,534& 2,982& 3,528& 4,200\\
     8& 14,832& 16,952& 19,424& 22,320& 25,728& 29,760\\
     \hline
   \end{tabular}
  \caption{Some values for ${d^\ast}^k_n$.}\label{tab2}
\end{table}

\begin{table}[h!]
  \centering
  \begin{tabular}{|c|r|r|r|r|r|r|}
     \hline
      \hspace{0.37cm}$n$\hspace{0.37cm} & \hspace{0.18cm}${c^\ast}_n$\hspace{0.18cm} & \hspace{0.38cm}${c^\ast}^2_n$\hspace{0.38cm} & \hspace{0.38cm}${c^\ast}^3_n$\hspace{0.38cm} & \hspace{0.38cm}${c^\ast}^4_n$\hspace{0.38cm} & \hspace{0.38cm}${c^\ast}^5_n$\hspace{0.38cm} & \hspace{0.38cm}${c^\ast}^6_n$\hspace{0.38cm} \\
     \hline
     1& & & & & &\\
     2& 0& & & & &\\
     3& 1& 1& & & &\\
     4& 2& 3& 4& & &\\
     5& 9& 11& 14& 18& &\\
     6& 44& 53& 64& 78& 96&\\
     7& 265& 309& 362& 426& 504& 600\\
     8& 1,854& 2,119& 2,428& 2,790& 3,216& 3,720\\
     \hline
   \end{tabular}
  \caption{Some values for ${c^\ast}^k_n$.}\label{tab3}
\end{table}

\begin{table}[h!]
  \centering
  \begin{tabular}{|c|r|r|r|r|r|r|}
     \cline{2-7}
      \multicolumn{1}{c|}{} & \hspace{0.1cm}$k=1$\hspace{0.1cm} & \hspace{0.2cm}$k=2$\hspace{0.2cm} & \hspace{0.2cm}$k=3$\hspace{0.2cm} & \hspace{0.2cm}$k=4$\hspace{0.2cm} & \hspace{0.2cm}$k=5$\hspace{0.2cm} & \hspace{0.2cm}$k=6$\hspace{0.2cm} \\
     \hline
     $n=2$& 0& & & & &\\
     $n=3$& 3& 3& & & &\\
     $n=4$& 8& 8& 8& & &\\
     $n=5$& 45& 45& 45& 45& &\\
     $n=6$& 264& 270& 240& 270& 264&\\
     $n=7$& 1,855& 1,855& 1,855& 1,855& 1,855& 1,855\\
     \hline
   \end{tabular}
  \caption{Some values for $D^k_n$.}\label{tab4}
\end{table}

\begin{table}[h!]
  \centering
  \begin{tabular}{|c|r|r|r|r|r|r|}
     \cline{2-7}
      \multicolumn{1}{c|}{} & \hspace{0.1cm}$k=1$\hspace{0.1cm} & \hspace{0.2cm}$k=2$\hspace{0.2cm} & \hspace{0.2cm}$k=3$\hspace{0.2cm} & \hspace{0.2cm}$k=4$\hspace{0.2cm} & \hspace{0.2cm}$k=5$\hspace{0.2cm} & \hspace{0.2cm}$k=6$\hspace{0.2cm} \\
     \hline
     $n=2$& 0& & & & &\\
     $n=3$& 3& 3& & & &\\
     $n=4$& 4& 8& 4& & &\\
     $n=5$& 40& 40& 40& 40& &\\
     $n=6$& 216& 234& 192& 234& 216&\\
     $n=7$& 1,603& 1,603& 1,603& 1,603& 1,603& 1,603\\
     \hline
   \end{tabular}
  \caption{Some values for ${D^\ast}^k_n$.}\label{tab5}
\end{table}

\begin{table}[h!]
  \centering
  \begin{tabular}{|c|r|r|r|r|r|r|}
     \cline{2-7}
      \multicolumn{1}{c|}{} & \hspace{0.1cm}$k=1$\hspace{0.1cm} & \hspace{0.2cm}$k=2$\hspace{0.2cm} & \hspace{0.2cm}$k=3$\hspace{0.2cm} & \hspace{0.2cm}$k=4$\hspace{0.2cm} & \hspace{0.2cm}$k=5$\hspace{0.2cm} & \hspace{0.2cm}$k=6$\hspace{0.2cm} \\
     \hline
     $n=2$& 0& & & & &\\
     $n=3$& 1& 1& & & &\\
     $n=4$& 1& 2& 1& & &\\
     $n=5$& 8& 8& 8& 8& &\\
     $n=6$& 36& 39& 32& 39& 36&\\
     $n=7$& 229& 229& 229& 229& 229& 229\\
     \hline
   \end{tabular}
  \caption{Some values for ${C^\ast}^k_n$.}\label{tab6}
\end{table}
{\color{white}
.

.

.

.

.

.
}
\end{document}